\documentclass[11pt,reqno]{amsart}
\usepackage[T1]{fontenc}
\usepackage[utf8]{inputenc}
\usepackage{lmodern}
\usepackage{amsmath,amssymb,mathtools}
\usepackage{microtype}
\usepackage[margin=1in]{geometry}
\usepackage{enumitem}
\usepackage[colorlinks=true,linkcolor=blue,citecolor=blue,urlcolor=blue]{hyperref}
\hypersetup{pdftitle={Tight sigma-filtrations and the ball fixed point property of C(omega*)},pdfauthor={Antonio Aviles and Mauricio Rossetto}}
\newtheorem{theorem}{Theorem}[section]
\newtheorem{lemma}[theorem]{Lemma}
\newtheorem{proposition}[theorem]{Proposition}
\newtheorem{corollary}[theorem]{Corollary}
\newtheorem{question}[theorem]{Question}
\newtheorem{mainthm}{Theorem}

\theoremstyle{definition}
\newtheorem{definition}[theorem]{Definition}
\theoremstyle{remark}
\newtheorem*{remark}{Remark}

\newcommand{\Pow}{\mathcal P}
\newcommand{\Cc}{\operatorname{Cc}}
\newcommand{\fin}{\mathrm{fin}}
\newcommand{\cont}{\mathfrak c}
\newcommand{\BB}{\mathfrak B}
\newcommand{\FF}{\mathcal F}
\newcommand{\gen}[1]{\langle #1\rangle}
\newcommand{\id}{\operatorname{id}}
\newcommand{\MA}{\mathrm{MA}}
\newcommand{\WFN}{\mathrm{WFN}}
\newcommand{\Kcc}{K^{\omega,\omega_1}}
\newcommand{\Xcc}{X^{\omega,\omega_1}}
\newcommand{\res}[2]{#1\mathbin{\upharpoonright}#2}
\newcommand{\below}[2]{#1\mathbin{|}_{#2}}

\newcommand{\axiomTSF}{\textup{(TSF)}}
\setlist[enumerate,1]{label=\textup{(\alph*)},leftmargin=*,itemsep=2pt}
\title[Tight $\sigma$-filtrations and the BFPP of $C(\omega^*)$]{Tight $\sigma$-filtrations and the ball\\ fixed point property of $C(\omega^*)$}
\author{Antonio Avil\'es}
\address{Departamento de Matem\'aticas, Universidad de Murcia, Campus de Espinardo, 30100 Murcia, Spain}
\email{avileslo@um.es}
\thanks{The first author was supported by MICIU/AEI/10.13039/501100011033 and ERDF---A way of making Europe (project PID2021-122126NB-C32), and by Fundaci\'on S\'eneca---ACyT Regi\'on de Murcia (project 21955/PI/22).}
\author{Maur\'icio Rossetto}
\address{Departamento de Matem\'atica, Instituto de Matem\'atica, Estat\'istica e Ci\^encia da Computa\c c\~ao, Universidade de S\~ao Paulo, Rua do Mat\~ao, 1010, 05508-090 S\~ao Paulo, Brazil}
\email{rossetto@ime.usp.br}
\thanks{The second author was supported by the S\~ao Paulo Research Foundation (FAPESP), grant 2025/08065-1.}
\date{October 8, 2026}
\begin{document}
\begin{abstract}
Assuming that $\Pow(\omega)$ is tightly $\sigma$-filtered, we prove that $C(\omega^*)$ fails the ball fixed point property. This axiom holds under the continuum hypothesis and in the $\aleph_2$-Cohen model. The proof shows that the Stone space of the countable--cocountable algebra $\Cc(\omega_1)$ is a continuous retract of $\omega^*$ and uses the equivalence between tight $\sigma$-filteredness of $\Cc(\omega_1)$ and of $\Pow(\omega)$. We also prove that the quotient of a tightly $\sigma$-filtered Boolean algebra $B$ with the countable separation property by an $\aleph_1$-generated $P$-ideal is a Boolean retract of $B$.
\end{abstract}
\maketitle

\section{Introduction}
A mapping $T\colon D\to Z$ defined on a subset $D$ of a Banach space $Z$ is \emph{nonexpansive} if $\|T(x)-T(y)\|\leq\|x-y\|$ for all $x,y\in D$. Since the theorems of Browder~\cite{Browder}, G\"ohde~\cite{Gohde} and Kirk~\cite{Kirk}, which provide fixed points for nonexpansive selfmaps of bounded, closed and convex subsets of uniformly convex spaces, it has been known that the existence of fixed points for such mappings is closely tied to the geometry of the underlying space; see~\cite{GoebelKirk}. Here we are interested in the case where the domain is the closed unit ball. All spaces of continuous functions are real Banach spaces.

\begin{definition}
A Banach space $X$ has the \emph{ball fixed point property} (BFPP) if every nonexpansive mapping $T\colon B_X\to B_X$ has a fixed point.
\end{definition}

Among spaces of continuous functions the BFPP may or may not hold. On the unit ball of the space $c$ of convergent sequences, that is, $C(\omega+1)$, the map
\[
(x_1,x_2,\ldots)\longmapsto (1,-1,x_1,x_2,\ldots)
\]
is an isometry without fixed points, whereas $\ell_\infty=C(\beta\omega)$ has the BFPP~\cite{Soardi}. This led Avil\'es et al.~\cite{BFPP} to ask for which compact Hausdorff spaces $K$ the space $C(K)$ has the BFPP. They proved the implications
\[
K\text{ is extremally disconnected}
\ \Longrightarrow\ C(K)\text{ has the BFPP}
\ \Longrightarrow\ K\text{ is an }F\text{-space}.
\]
The converse of the first implication is an open problem, while the converse of the second one is false; see~\cite{BFPP}. The natural place to test where the BFPP lies between these two classes is $\omega^*=\beta\omega\setminus\omega$, a classical example of a compact $F$-space which is not extremally disconnected. The main result of~\cite{BFPP} about it is that, under the Continuum Hypothesis (CH), $C(\omega^*)$ fails the BFPP.

The argument has two independent parts. The first is an example of a space of continuous functions that does not satisfy the BFPP in ZFC. Let
\[
\Xcc:=\bigl\{f\in\ell_\infty(\omega_1):
\res f{\omega_1\setminus A}\text{ is constant for some countable }A\subseteq\omega_1\bigr\}.
\]
This is a closed subalgebra of $\ell_\infty(\omega_1)$ containing the constants, so, by Arens' representation theorem~\cite{Arens}, it is isometrically isomorphic to $C(\Kcc)$ for some compact Hausdorff space $\Kcc$. Its idempotents are the indicator functions of countable and cocountable subsets of $\omega_1$, so $\Kcc$ is the Stone space of the countable--cocountable algebra $\Cc(\omega_1)$. It is shown in~\cite{BFPP} that $C(\Kcc)$ fails the BFPP. The second part transfers this failure to $\omega^*$, and this is where CH is used, through the following fact; cf.~\cite[Theorems 1.4.4 and 1.8.1]{vanMill}.

\begin{theorem}[CH]\label{thm:CH}
Every zero-dimensional compact $F$-space of weight $\mathfrak{c}$ is a continuous retract of $\omega^*$.
\end{theorem}

Under CH, $\Kcc$ is such a space, hence a retract of $\omega^*$. On the other hand, if $L$ is a continuous retract of $K$ and $C(K)$ has the BFPP, then so does $C(L)$~\cite{BFPP}. Therefore $C(\omega^*)$ cannot have the BFPP under CH. In this paper, we generalize this argument under the following set-theoretic assumption:
\begin{equation}\tag{TSF}\label{eq:TSF}
\Pow(\omega)\text{ is tightly }\sigma\text{-filtered}.
\end{equation}
This axiom holds under CH and in the $\aleph_2$-Cohen model, obtained by adding $\aleph_2$ Cohen reals to a model of CH; see~\cite{Koppelberg,Geschke,Steprans}. We recall the definition in Section~\ref{sec:prelim}.
This topic can be found in the literature under different terminologies. Tight $\sigma$-filteredness was introduced by Koppelberg~\cite{Koppelberg} and developed by Geschke~\cite{Geschke}, who characterized it using additive $\sigma$-skeletons. Avil\'es and Brech~\cite{AB} formulated the same notion through iterated push-outs over countable subalgebras. The subalgebras called \emph{good} by Stepr\=ans~\cite{Steprans}, \emph{$\aleph_0$-ideal} by Dow and Hart~\cite{DH}, and \emph{$\sigma$-subalgebras} in~\cite{Geschke,AB} express the same countable generation condition on cuts. In the language of Dow and Hart~\cite{DH}, the axiom can also be formulated as
\[
\axiomTSF\quad\Longleftrightarrow\quad
\cont\leq\aleph_2\text{ and }\Pow(\omega)/\fin
\text{ is Cohen--Parovi\v cenko}.
\]
We justify this equivalence after Proposition~\ref{prop:identification}. The quotient by $\fin$ and the cardinal bound are essential here: $\Pow(\omega)$ has atoms, whereas Cohen--Parovi\v cenko algebras are atomless, and $\Pow(\omega)/\fin$ is Cohen--Parovi\v cenko even in Cohen models with larger continuum~\cite[Proposition 5.4]{DH}.

Another formulation involves the \emph{weak Freese--Nation property} (WFN): a Boolean algebra $B$ has this property if there is $f\colon B\to[B]^{\leq\aleph_0}$ such that, whenever $a\leq b$, some $d\in f(a)\cap f(b)$ satisfies $a\leq d\leq b$. For $\Pow(\omega)$, Geschke's results give
\[
\axiomTSF\quad\Longleftrightarrow\quad
\WFN\bigl(\Pow(\omega)\bigr)\text{ and }\cont\leq\aleph_2;
\]
see~\cite[Corollary 5.5]{Geschke}. Further consequences of WFN were developed by Fuchino, Koppelberg and Shelah~\cite{FKS}, and by Fuchino, Geschke and Soukup~\cite{FGS}.

Parovi\v cenko's theorem~\cite{Parovicenko} characterizes $\Pow(\omega)/\fin$ under CH. For regular $\cont$, Avil\'es and Brech~\cite{AB} constructed a canonical tightly $\sigma$-filtered Boolean algebra $\BB$ of size $\cont$. Under~\eqref{eq:TSF}, the continuum is automatically regular and $\BB\cong\Pow(\omega)/\fin$; we spell out the deduction from the literature in Proposition~\ref{prop:identification}. We show that every tightly $\sigma$-filtered algebra of size at most $\cont$ with the countable separation property is a retract of $\BB$, and that $\Cc(\omega_1)$ is tightly $\sigma$-filtered exactly when $\Pow(\omega)$ is. This yields our first main result.

\begin{mainthm}\label{thm:A}
Assume that $\Pow(\omega)$ is tightly $\sigma$-filtered. Then $\Kcc$ is a continuous retract of $\omega^*$. Consequently, $C(\omega^*)$ fails the ball fixed point property. In particular, these conclusions hold in the $\aleph_2$-Cohen model.
\end{mainthm}

In contrast, under $\MA+\neg\mathrm{CH}$ the algebra $\Cc(\omega_1)$ is not tightly $\sigma$-filtered; see Corollary~\ref{cor:MA}. Whether $C(\omega^*)$ can have the BFPP in some model remains open.

The proof of Theorem~\ref{thm:CH} in~\cite{vanMill} uses the fact that, under CH, nonempty closed $P$-sets of $\omega^*$ are retracts of $\omega^*$. Under Stone duality, a closed subset of the Stone space of $B$ corresponds to a quotient $B/I$, and being a $P$-set corresponds to $I$ being a $P$-ideal. Our second result is an analogue for tightly $\sigma$-filtered algebras.

\begin{mainthm}\label{thm:B}
Let $B$ be a tightly $\sigma$-filtered Boolean algebra with the countable separation property. If $I$ is a proper $\aleph_1$-generated $P$-ideal of $B$, then the quotient map $B\to B/I$ has a Boolean homomorphic right inverse. In particular, $B/I$ is a Boolean retract of $B$.
\end{mainthm}

Section~\ref{sec:prelim} collects the background on tightly $\sigma$-filtered Boolean algebras and their extensions. Section~\ref{sec:main} contains the retraction result, the analysis of $\Cc(\omega_1)$ and the proof of Theorem~\ref{thm:A}. Section~\ref{sec:quotients} is devoted to Theorem~\ref{thm:B}. Section~\ref{sec:questions} lists the remaining questions.

\section{Preliminaries}\label{sec:prelim}
Throughout the paper, $\gen X$ denotes the subalgebra generated by a subset $X$ of a Boolean algebra. Given a subalgebra $A\subseteq B$, we say that $B$ is \emph{countably generated over $A$} if $B=\gen{A\cup Q}$ for some countable $Q\subseteq B$. Recall from~\cite{AB} that an embedding $A\to B$ is a \emph{posex} (push-out separable extension) if it fits into a push-out diagram of embeddings over countable subalgebras. When $A\subseteq B$, we also say that $B$ is a posex of $A$.

A topological space $L$ is a \emph{continuous retract} of $K$ if there are continuous maps $e\colon L\to K$ and $r\colon K\to L$ with $r\circ e=\id_L$. Dually, a Boolean algebra $A$ is a \emph{Boolean retract} of $B$ if there are homomorphisms $e\colon A\to B$ and $r\colon B\to A$ with $r\circ e=\id_A$.

Given $b\in B$ and a subalgebra $A\subseteq B$, write
\[
\below A b=\{a\in A:a\leq b\}.
\]
When $b\in A$, this principal ideal is itself a Boolean algebra, with unit $b$ and relative complement $x\mapsto b\wedge -x$. We say that $A$ is a \emph{$\sigma$-subalgebra} of $B$, written $A\leq_\sigma B$, if $\below A b$ is countably generated for every $b\in B$.

We say that $B$ is \emph{tightly $\sigma$-filtered over $A$} if there exists a continuous increasing chain $(B_\alpha)_{\alpha\leq\lambda}$ of subalgebras, starting at $B_0=A$ and ending at $B_\lambda=B$, such that each successor step $B_\alpha\to B_{\alpha+1}$ is a posex. An algebra is \emph{tightly $\sigma$-filtered} if it is tightly $\sigma$-filtered over $\{0,1\}$.

\begin{definition}\label{def:skeleton}
An \emph{additive $\sigma$-skeleton} of a Boolean algebra $B$ is a family $\FF$ of subalgebras of $B$ such that:
\begin{enumerate}
\item $\{0,1\}\in\FF$;
\item $\gen{\bigcup\mathcal G}\in\FF$ for every $\mathcal G\subseteq\FF$;
\item every infinite subalgebra $X\subseteq B$ is contained in some $X'\in\FF$ with $|X'|=|X|$;
\item every member of $\FF$ is a $\sigma$-subalgebra of $B$.
\end{enumerate}
\end{definition}

\begin{proposition}[Avil\'es and Brech~\cite{AB}]\label{prop:posex}
Let $A\subseteq B$ be Boolean algebras.
\begin{enumerate}
\item $B$ is a posex of $A$ if and only if $B$ is countably generated over $A$ and $A\leq_\sigma B$.
\item $A\leq_\sigma B$ if and only if every subalgebra $C$ with $A\subseteq C\subseteq B$ which is countably generated over $A$ is a posex of $A$.
\item If $B$ is countable, then $B$ is a posex of $A$.
\item If $B$ is a posex of $A$ and $A\subseteq C\subseteq B$, with $C$ countably generated over $A$, then $C$ is a posex of $A$ and $B$ is a posex of $C$.
\item If $L$ and $M$ are countably generated ideals of $A$ such that $\ell\wedge m=0$ for all $\ell\in L$ and $m\in M$, then $A$ has a posex $A(z)=\gen{A\cup\{z\}}$ with $\below A z=L$ and $\below A{-z}=M$.
\item If $A_0\subseteq A_1\subseteq\cdots$ are $\sigma$-subalgebras of $B$, then $\bigcup_{n<\omega}A_n\leq_\sigma B$.
\item $B$ is tightly $\sigma$-filtered if and only if it has an additive $\sigma$-skeleton. This characterization is due to Geschke~\cite{Geschke}.
\item If $\FF$ is an additive $\sigma$-skeleton of $B$, $C\in\FF$, and $C\subseteq C'\subseteq B$ with $C'$ countably generated over $C$, then there exists $C''\in\FF$ with $C'\subseteq C''$ and $C''$ countably generated over $C'$.
\end{enumerate}
\end{proposition}

\begin{proposition}\label{prop:concatenate}
Let $\theta$ be a limit ordinal and $(B_\alpha)_{\alpha<\theta}$ an increasing sequence of Boolean algebras with $B_0=\{0,1\}$. Suppose that for every $0<\alpha<\theta$, $B_\alpha$ is tightly $\sigma$-filtered over $\bigcup_{\beta<\alpha}B_\beta$. Then $\bigcup_{\alpha<\theta}B_\alpha$ is tightly $\sigma$-filtered. In particular, if $A$ is tightly $\sigma$-filtered and $B$ is tightly $\sigma$-filtered over $A$, then $B$ is tightly $\sigma$-filtered.
\end{proposition}
\begin{proof}
Concatenate the witnessing filtrations, taking unions at limit stages.
\end{proof}

\begin{theorem}[Avil\'es and Brech~\cite{AB}]\label{thm:AB}
Assume that $\cont$ is regular. There exists a unique, up to isomorphism, tightly $\sigma$-filtered Boolean algebra $\BB$ of size $\cont$ such that, for every posex extension $A\to D$ with $|A|<\cont$, every embedding $A\to\BB$ extends to an embedding $D\to\BB$.

In particular, every tightly $\sigma$-filtered Boolean algebra of size at most $\cont$ embeds into $\BB$.
\end{theorem}

\begin{definition}\label{def:CSP}
A Boolean algebra $B$ has the \emph{countable separation property} (CSP) if, for all countable $X,Y\subseteq B$ with $x\leq y$ for every $x\in X$ and $y\in Y$, there exists $z\in B$ with $x\leq z\leq y$ for all such $x,y$.
\end{definition}

The following consequence of the known results makes explicit why~\eqref{eq:TSF} suffices in Theorem~\ref{thm:A}.
\begin{proposition}\label{prop:identification}
If $\Pow(\omega)$ is tightly $\sigma$-filtered, then $\cont\in\{\aleph_1,\aleph_2\}$ and $\Pow(\omega)/\fin\cong\BB$.
\end{proposition}
\begin{proof}
By~\cite[Corollaries 5.4 and 5.5]{Geschke}, $\cont\leq\aleph_2$ and $\Pow(\omega)$ has WFN. In particular, $\cont$ is regular. Since $\fin$ is countable, WFN passes between $\Pow(\omega)$ and $\Pow(\omega)/\fin$; an algebra of size at most $\aleph_2$ with WFN is tightly $\sigma$-filtered. Thus $\Pow(\omega)/\fin$ is tightly $\sigma$-filtered, as also explained in the proof of~\cite[Theorem 6.3]{Geschke}.

The result of Fuchino, Geschke and Soukup~\cite{FGS} recorded in~\cite[Lemma 6.1]{Geschke} says that WFN of $\Pow(\omega)$, together with $\cont<\aleph_\omega$, implies $\MA(\mathrm{countable})$. Here this denotes Martin's axiom for countable partial orders and families of fewer than $\cont$ dense sets. By~\cite[Proposition 2.12]{DH}, this gives $\mathfrak m_c(\Pow(\omega)/\fin)=\cont$: every subalgebra of size less than $\cont$ is $\aleph_0$-ideal complete. By~\cite[Proposition 10]{AB}, this is exactly the extension property for posexes appearing in Theorem~\ref{thm:AB}. Since $|\Pow(\omega)/\fin|=\cont$, that theorem yields the required isomorphism.
\end{proof}

\begin{remark}
For completeness, we verify the Cohen--Parovi\v cenko formulation of~\eqref{eq:TSF} used in the introduction. As observed in~\cite[Section 2.5]{AB}, the algebra $\BB$ of Theorem~\ref{thm:AB} is Cohen--Parovi\v cenko in the sense of~\cite[Definition 5.2]{DH}. Proposition~\ref{prop:identification} therefore gives the forward implication. Conversely, suppose that $\cont\leq\aleph_2$ and $\Pow(\omega)/\fin$ is Cohen--Parovi\v cenko. Since $\cont$ is regular, $\BB$ exists, and uniqueness of Cohen--Parovi\v cenko algebras of size $\cont$ under this cardinal bound~\cite[Theorem 5.5]{DH} gives $\Pow(\omega)/\fin\cong\BB$. Thus $\Pow(\omega)/\fin$ is tightly $\sigma$-filtered. Lifting its filtration to $\Pow(\omega)$, as in the proof of~\cite[Theorem 11]{AB}, yields~\eqref{eq:TSF}.
\end{remark}

\section{\texorpdfstring{The BFPP of $C(\omega^*)$ under tight $\sigma$-filteredness}{The BFPP of C(omega*) under tight sigma-filteredness}}\label{sec:main}
We first prove an extension lemma for homomorphisms into algebras with the CSP.

\begin{proposition}\label{prop:extend}
Let $C\subseteq D$ be Boolean algebras, with $D$ a posex of $C$, and let $A$ satisfy the CSP. Every homomorphism $h\colon C\to A$ extends to a homomorphism $\bar h\colon D\to A$.
\end{proposition}
\begin{proof}
Write $D=\gen{C\cup\{s_1,s_2,\ldots\}}$ and put $D_0=C$ and $D_{n+1}=\gen{D_n\cup\{s_{n+1}\}}$. It suffices to extend a homomorphism $h_n\colon D_n\to A$ to $D_{n+1}$ and take the union.

By Proposition~\ref{prop:posex}(d), $D$ is a posex of $D_n$, so the ideal
\[
I_n=\{d\in D_n:d\leq s_{n+1}\}
\]
and the filter $F_n=\{d\in D_n:d\geq s_{n+1}\}$ have countable cofinal and coinitial families, respectively. Choose such families $(a_k)_{k<\omega}$ and $(b_k)_{k<\omega}$. Since $a_k\leq s_{n+1}\leq b_j$, we have $h_n(a_k)\leq h_n(b_j)$ for all $k,j$. The CSP provides $c\in A$ with
\[
h_n(a_k)\leq c\leq h_n(b_j)\qquad(k,j<\omega).
\]
Consequently, $d\leq s_{n+1}$ implies $h_n(d)\leq c$, and $d\geq s_{n+1}$ implies $h_n(d)\geq c$. The one-generator form of Sikorski's extension theorem therefore extends $h_n$ to $D_{n+1}$ by sending $s_{n+1}$ to $c$.
\end{proof}

The next lemma replaces the weight bound in Theorem~\ref{thm:CH} by tight $\sigma$-filteredness. Its ambient algebra is $\BB$, which exists whenever $\cont$ is regular.

\begin{lemma}\label{lem:retract}
Assume that $\cont$ is regular. Every tightly $\sigma$-filtered Boolean algebra $A$ with the CSP and $|A|\leq\cont$ is a Boolean retract of $\BB$.
\end{lemma}
\begin{proof}
We use the push-out construction of~\cite{AB}, starting with $B_0=A$ instead of $\{0,1\}$. Construct a continuous chain $(B_\alpha)_{\alpha\leq\cont}$ in which each successor is the push-out of $B_\alpha$ and a countable algebra $S$ over a countable subalgebra $R\subseteq B_\alpha$. The bookkeeping arranges that every pair $R\subseteq S$, with $R\subseteq B_\alpha$, is treated at a later stage. This is possible because $|B_\alpha|\leq\cont$ and $\cont^{\aleph_0}=\cont$. More explicitly, partition $\cont$ into sets $\Phi_\alpha$ of size $\cont$, with $\min\Phi_\alpha\geq\alpha$, and use the stages in $\Phi_\alpha$ to enumerate the countable extensions over countable subalgebras of $B_\alpha$. These extensions may be chosen to ensure that the final algebra $B'=B_\cont$ has size $\cont$.

Every successor step is a posex, so $B'$ is tightly $\sigma$-filtered over $A$. By Proposition~\ref{prop:concatenate}, it is tightly $\sigma$-filtered. Given a posex $C\to D$ with $|C|<\cont$ and an embedding $C\to B'$, regularity of $\cont$ places the image of $C$ inside some $B_\alpha$. At a later stage the bookkeeping adds the countable push-out data defining $C\to D$, and hence extends the embedding to $D$. Thus $B'$ satisfies Theorem~\ref{thm:AB}, and $B'\cong\BB$.

To retract $B'$ onto $A$, recursively define coherent homomorphisms $r_\alpha\colon B_\alpha\to A$. Start with $r_0=\id_A$. Proposition~\ref{prop:extend} extends $r_\alpha$ across each successor posex, and at limit stages take unions. The resulting $r_\cont\colon B'\to A$ restricts to the identity on $A$.
\end{proof}

\begin{lemma}\label{lem:CcCSP}
The algebra $\Cc(\omega_1)$ satisfies the CSP.
\end{lemma}
\begin{proof}
It is closed under countable unions: a countable union of countable sets is countable, and a union containing a cocountable set is cocountable. Given countable $X,Y\subseteq\Cc(\omega_1)$ with $x\subseteq y$ for all $x\in X$, $y\in Y$, the set $z=\bigcup X$ belongs to $\Cc(\omega_1)$ and separates $X$ and $Y$.
\end{proof}

We next analyze tight $\sigma$-filteredness of $\Cc(\omega_1)$.

\begin{lemma}\label{lem:stages}
If $(B_\xi)_{\xi\leq\lambda}$ is a tight $\sigma$-filtration of $A$, then $B_\xi\leq_\sigma A$ for every $\xi\leq\lambda$.
\end{lemma}
\begin{proof}
Fix $\xi\leq\lambda$. We prove by induction on $\eta\in[\xi,\lambda]$ that $\below{B_\xi}x$ is countably generated for every $x\in B_\eta$. For $\eta=\xi$ the ideal is generated by $x$.

If $x\in B_{\eta+1}$, Proposition~\ref{prop:posex}(a) gives a countable cofinal set $G\subseteq\below{B_\eta}x$. For every $g\in G$, the induction hypothesis gives a countable cofinal $H_g\subseteq\below{B_\xi}g$. The countable set $\bigcup_{g\in G}H_g$ is cofinal in $\below{B_\xi}x$: if $c\in B_\xi$ and $c\leq x$, choose $g\in G$ with $c\leq g$, and then $h\in H_g$ with $c\leq h$.

At a limit stage, continuity places $x$ in an earlier algebra, to which the induction hypothesis applies.
\end{proof}

\begin{proposition}\label{prop:relativetight}
If $A$ is tightly $\sigma$-filtered and $b\in A$, then the relative Boolean algebra $\below A b$ is tightly $\sigma$-filtered.
\end{proposition}
\begin{proof}
The case $b=0$ is immediate. Let $(B_\xi)_{\xi\leq\lambda}$ be a tight $\sigma$-filtration of $A$, let $\pi_b\colon A\to\below A b$ be given by $\pi_b(x)=x\wedge b$, and put $D_\xi=\pi_b[B_\xi]$. This is an increasing continuous chain from $\{0,b\}$ to $\below A b$. If $B_{\xi+1}=\gen{B_\xi\cup Q}$ with $Q$ countable, then
\[
D_{\xi+1}=\gen{D_\xi\cup\pi_b[Q]}.
\]
It remains to prove $D_\xi\leq_\sigma D_{\xi+1}$. In fact, $D_\xi\leq_\sigma\below A b$.

For $x\in A$ and $y\leq b$, Boolean algebra identities give
\[
\pi_b(x)\leq y\quad\Longleftrightarrow\quad x\leq y\vee -b.
\]
By Lemma~\ref{lem:stages}, choose a countable cofinal set $G\subseteq\below{B_\xi}{y\vee -b}$. Its image $\pi_b[G]$ is cofinal in $\below{D_\xi}y$: if $e=\pi_b(x)\leq y$ with $x\in B_\xi$, choose $g\in G$ above $x$, and then $e\leq\pi_b(g)$. Proposition~\ref{prop:posex}(a) completes the proof.
\end{proof}

\begin{corollary}\label{cor:largec}
If $\cont\geq\aleph_3$, then $\Cc(\omega_1)$ is not tightly $\sigma$-filtered.
\end{corollary}
\begin{proof}
If $\Cc(\omega_1)$ were tightly $\sigma$-filtered, Proposition~\ref{prop:relativetight} would give tight $\sigma$-filteredness of
\[
\below{\Cc(\omega_1)}\omega=\Pow(\omega).
\]
This contradicts Geschke's theorem that no complete Boolean algebra of size greater than $\aleph_2$ is tightly $\sigma$-filtered~\cite[Corollary 5.4]{Geschke}.
\end{proof}

The converse direction requires the following relative filtration lemma.

\begin{lemma}\label{lem:relative}
Let $B$ be tightly $\sigma$-filtered, and let $e_0\leq e_1\leq\cdots$ be elements of $B$. Then $B$ is tightly $\sigma$-filtered over
\[
A:=\gen{\bigcup_{n<\omega}\below B{e_n}}.
\]
\end{lemma}
\begin{proof}
If $B$ is countable, Proposition~\ref{prop:posex}(c) applies. Otherwise, choose an additive $\sigma$-skeleton $\FF$ of $B$ and enumerate $B=\{b_\xi:\xi<\kappa\}$. By Definition~\ref{def:skeleton}, choose a continuous increasing chain $(S_\xi)_{\xi\leq\kappa}$ in $\FF$ such that $S_0$ is countable and contains all $e_n$, and
\[
S_{\xi+1}=\gen{S_\xi\cup T_\xi},
\]
where $T_\xi\in\FF$ is countable and contains $b_\xi$. Then $S_\kappa=B$. Put $C_\xi=\gen{A\cup S_\xi}$. The chain $(C_\xi)$ is continuous, $C_0$ is countably generated over $A$, and $C_{\xi+1}$ is countably generated over $C_\xi$. We show that $A$ and every $C_\xi$ are $\sigma$-subalgebras of $B$.

First,
\[
\gen{\below B{e_n}}=\below B{e_n}\ \cup\ \{-x:x\in\below B{e_n}\}.
\]
For $b\in B$, the ideal $\below{\gen{\below B{e_n}}}b$ is principal: its generator is $b$ if $-e_n\leq b$, and $b\wedge e_n$ otherwise. Thus $\gen{\below B{e_n}}\leq_\sigma B$. These subalgebras form an increasing sequence whose union is $A$, so Proposition~\ref{prop:posex}(f) gives $A\leq_\sigma B$.

Now let $S\in\FF$ contain every $e_n$ and put $D_n=\gen{\below B{e_n}\cup S}$. Since $e_n\in S$, one has
\[
D_n=\{x\vee y:x\in\below B{e_n},\ y\in\below S{-e_n}\}.
\]
Indeed, the family on the right contains the generators and is closed under joins and complements; the relative decomposition of the complement is
\[
-(x\vee y)=(e_n\wedge -x)\vee(-e_n\wedge -y).
\]
Given $b\in B$, use $S\leq_\sigma B$ to choose a countable cofinal $Y\subseteq\below S{b\wedge -e_n}$. Then
\[
\{(b\wedge e_n)\vee y:y\in Y\}
\]
is cofinal in $\below{D_n}b$. Consequently $D_n\leq_\sigma B$, and another application of Proposition~\ref{prop:posex}(f) gives
\[
\gen{A\cup S}=\bigcup_{n<\omega}D_n\leq_\sigma B.
\]
Taking $S=S_\xi$ proves the claim for $C_\xi$. Proposition~\ref{prop:posex}(a) now shows that $A\subseteq C_0\subseteq\cdots\subseteq C_\kappa=B$ is the required filtration.
\end{proof}

\begin{theorem}\label{thm:equivalence}
$\Cc(\omega_1)$ is tightly $\sigma$-filtered if and only if $\Pow(\omega)$ is tightly $\sigma$-filtered.
\end{theorem}
\begin{proof}
The forward implication follows from Proposition~\ref{prop:relativetight}, since $\below{\Cc(\omega_1)}\omega=\Pow(\omega)$.

Conversely, suppose $\Pow(\omega)$ is tightly $\sigma$-filtered. For each $\alpha<\omega_1$, put
\[
B_\alpha=\{x\in\Cc(\omega_1):x\subseteq\alpha\text{ or }\omega_1\setminus x\subseteq\alpha\}.
\]
These are increasing subalgebras, $B_0=\{\varnothing,\omega_1\}$, and their union is $\Cc(\omega_1)$ because every countable subset of $\omega_1$ is bounded. The map $x\mapsto x\cap(\alpha+1)$ is an isomorphism from $B_\alpha$ onto $\Pow(\alpha+1)$. Thus $B_\alpha$ is either finite or isomorphic to $\Pow(\omega)$, and in either case it is tightly $\sigma$-filtered.

Fix $0<\alpha<\omega_1$. Choose a nondecreasing sequence of ordinals $(e_n)_{n<\omega}$ below $\alpha$ such that every $\beta<\alpha$ satisfies $\beta\leq e_n$ for some $n$. When $\alpha=\gamma+1$, take $e_n=\gamma$ for all $n$; when $\alpha$ is a limit, take a cofinal sequence. Viewed as elements of $B_\alpha$, these satisfy
\[
\below{B_\alpha}{e_n}=\Pow(e_n),\qquad
\gen{\below{B_\alpha}{e_n}}^{B_\alpha}=B_{e_n}.
\]
Hence
\[
\gen{\bigcup_{n<\omega}\below{B_\alpha}{e_n}}^{B_\alpha}
=\bigcup_{n<\omega}B_{e_n}
=\bigcup_{\beta<\alpha}B_\beta.
\]
Lemma~\ref{lem:relative} shows that $B_\alpha$ is tightly $\sigma$-filtered over this union. Proposition~\ref{prop:concatenate} now applies. Notice that the original chain $(B_\alpha)$ need not be continuous; the intervening relative filtrations supply the missing stages.
\end{proof}

\begin{corollary}\label{cor:Cohen}
In the $\aleph_2$-Cohen model, $\Cc(\omega_1)$ is tightly $\sigma$-filtered.
\end{corollary}
\begin{proof}
In this model $\Pow(\omega)$ is tightly $\sigma$-filtered~\cite{Koppelberg,Geschke,Steprans}. Apply Theorem~\ref{thm:equivalence}.
\end{proof}

\begin{corollary}\label{cor:MA}
Under $\MA+\neg\mathrm{CH}$, $\Cc(\omega_1)$ is not tightly $\sigma$-filtered. In particular, there are models with $\cont=\aleph_2$ in which it is not tightly $\sigma$-filtered.
\end{corollary}
\begin{proof}
By Theorem~\ref{thm:equivalence}, it suffices to rule out tight $\sigma$-filteredness of $\Pow(\omega)$. If it held, Proposition~\ref{prop:identification} would make $\Pow(\omega)/\fin$ tightly $\sigma$-filtered. Such an algebra is $(\aleph_1,\aleph_0)$-ideal and hence has no strictly increasing chain of order type $\omega_2$~\cite[Proposition 2.5]{DH}. On the other hand, $\MA+\neg\mathrm{CH}$ gives such a chain in $\Pow(\omega)/\fin$: at each stage below $\omega_2$, Martin's axiom supplies an infinite pseudointersection of the fewer than $\cont$ earlier members of a decreasing chain; splitting it gives a strictly smaller member modulo finite sets. Taking complements produces an increasing chain. This is the obstruction also used in~\cite[Theorem 5.10]{DH}. Finally, $\MA+\cont=\aleph_2$ is consistent relative to ZFC.
\end{proof}

\begin{proof}[Proof of Theorem~\ref{thm:A}]
Assume~\eqref{eq:TSF}. Proposition~\ref{prop:identification} gives regularity of $\cont$ and $\BB\cong\Pow(\omega)/\fin$. Moreover,
\[
|\Cc(\omega_1)|=|[\omega_1]^{\leq\aleph_0}|=\cont:
\]
the lower bound follows from $\Pow(\omega)\subseteq\Cc(\omega_1)$, and the upper bound from $\omega_1^{\aleph_0}\leq\cont^{\aleph_0}=\cont$. By Theorem~\ref{thm:equivalence} and Lemma~\ref{lem:CcCSP}, $\Cc(\omega_1)$ is tightly $\sigma$-filtered and satisfies the CSP. Lemma~\ref{lem:retract} makes it a Boolean retract of $\BB$, hence of $\Pow(\omega)/\fin$.

Stone duality gives a continuous retraction of $\omega^*$ onto a copy of $\Kcc$. Since $C(\Kcc)$ fails the BFPP and the BFPP passes to spaces of continuous functions on retracts~\cite{BFPP}, $C(\omega^*)$ fails the BFPP. The Cohen-model conclusion follows because~\eqref{eq:TSF} holds there.
\end{proof}

\section{A Boolean analogue of a topological lemma}\label{sec:quotients}
The retractions used above suggest analogous questions for quotients of Boolean algebras. We begin with the classical topological statement used in the proof of Theorem~\ref{thm:CH}; see~\cite[Theorem 1.8.1]{vanMill}. We make the CH assumption explicit in this formulation.

\begin{lemma}[CH]\label{lem:Pset}
Every nonempty closed $P$-set of $\omega^*$ is a retract of $\omega^*$.
\end{lemma}

Here a closed set $X\subseteq K$ is a \emph{$P$-set} if every countable intersection of neighborhoods of $X$ is again a neighborhood of $X$. An ideal $I$ of a Boolean algebra is a \emph{$P$-ideal} if every countable subset of $I$ has an upper bound in $I$, with respect to the Boolean order. We call $I$ \emph{$\aleph_1$-generated} if it has a generating family of cardinality at most $\aleph_1$. All quotient ideals considered in this section are proper.

Under Stone duality, a closed subset $X$ of the Stone space of $B$ corresponds to the quotient by the ideal of clopen sets disjoint from $X$. The $P$-set condition is equivalent to this ideal being a $P$-ideal. This motivates the following question.

\begin{question}\label{q:quotient}
Let $B$ be a tightly $\sigma$-filtered Boolean algebra with the CSP, and let $I$ be a proper $P$-ideal of $B$. Is $B/I$ a Boolean retract of $B$?
\end{question}

The following lifting result is a special case of the theorem of Koppelberg~\cite{Koppelberg} recorded and generalized in~\cite[Theorem 1.9]{Geschke}. We include a proof in the notation of this paper.

\begin{lemma}[Koppelberg; see Geschke~\cite{Geschke}]\label{lem:lifting}
Let $B$ satisfy the CSP and let $I$ be a proper $P$-ideal. If $B/I$ is tightly $\sigma$-filtered, the quotient map $\pi\colon B\to B/I$ has a Boolean homomorphic right inverse. In particular, $B/I$ is a Boolean retract of $B$.
\end{lemma}
\begin{proof}
Put $A=B/I$. We first verify the extension step. Suppose $C\subseteq D\subseteq A$, $D$ is a posex of $C$, and $s\colon C\to B$ satisfies $\pi\circ s=\id_C$. We claim that $s$ extends to $\widetilde s\colon D\to B$ with $\pi\circ\widetilde s=\id_D$.

By enumerating countably many generators for $D$ over $C$ and using Proposition~\ref{prop:posex}(d), it suffices to treat $D=\gen{C\cup\{d\}}$. The ideals $\below C d$ and $\below C{-d}$ are generated by sequences $(\ell_n)_n$ and $(m_n)_n$. Choose $b_0\in B$ with $\pi(b_0)=d$ and put
\[
e_n=s(\ell_n)\wedge -b_0\in I,
\qquad
f_n=s(m_n)\wedge b_0\in I.
\]
Since $\ell_n\wedge m_k=0$, we have $s(\ell_n)\wedge s(m_k)=0$ for all $n,k$. In particular, $e_n\wedge s(m_k)=0$ and $f_n\wedge s(\ell_k)=0$.

Because $I$ is a $P$-ideal, choose $i\in I$ above all $e_n$. By the CSP, choose $c\in B$ with $e_n\leq c\leq-s(m_k)$ for all $n,k$. Then $u=i\wedge c\in I$ satisfies $e_n\leq u$ and $u\wedge s(m_k)=0$. Similarly, choose $v\in I$ with $f_n\leq v$ and $v\wedge s(\ell_k)=0$ for all $n,k$. Set
\[
b=(b_0\wedge-v)\vee u.
\]
Since $b\mathbin{\triangle}b_0\leq u\vee v\in I$, one has $\pi(b)=d$. Also, for every $n$,
\begin{align*}
s(\ell_n)&=(s(\ell_n)\wedge b_0\wedge-v)\vee e_n\leq b,\\
s(m_n)\wedge b&=(f_n\wedge-v)\vee(s(m_n)\wedge u)=0.
\end{align*}
The one-generator extension theorem therefore gives $\widetilde s\colon D\to B$ extending $s$ with $\widetilde s(d)=b$. The maps $\pi\circ\widetilde s$ and $\id_D$ agree on the generators $C\cup\{d\}$, so they agree on $D$. This proves the extension step.

Now take a tight $\sigma$-filtration $(A_\alpha)_{\alpha\leq\lambda}$ of $A$. Starting with the unique unital homomorphism $s_0\colon\{0,1\}\to B$, apply the extension step at successors and take unions at limits. The resulting homomorphism $s\colon A\to B$ satisfies $\pi\circ s=\id_A$. Thus $s$ is an embedding and $s\circ\pi$ is a retraction of $B$ onto $s[A]$.
\end{proof}

The remaining argument is a preservation result for quotients. It uses neither the CSP nor the $P$-ideal hypothesis. The cut calculation in the proof is also the one used by Dow and Hart~\cite[Lemma 5.11]{DH}.

\begin{proposition}\label{prop:quotient}
If $B$ is tightly $\sigma$-filtered and $I$ is a proper $\aleph_1$-generated ideal of $B$, then $B/I$ is tightly $\sigma$-filtered.
\end{proposition}
\begin{proof}
Let $G\subseteq I$ generate $I$, with $|G|\leq\aleph_1$, and let $\pi\colon B\to B/I$ be the quotient map. Choose an additive $\sigma$-skeleton $\FF$ of $B$. By Definition~\ref{def:skeleton}(c), there is $H\in\FF$ containing $G$ with $|H|\leq\aleph_1$; for finite $B$ simply take $H=B$.

We claim that, whenever $C\in\FF$ and $H\subseteq C$, one has $\pi[C]\leq_\sigma B/I$. In fact, for every $b\in B$,
\[
\below{\pi[C]}{\pi(b)}=\pi[\below C b].
\]
The inclusion from right to left is immediate. Conversely, if $c\in C$ and $\pi(c)\leq\pi(b)$, then $c\wedge-b\in I$. Since $G\subseteq C$ generates $I$, choose $i\in I\cap C$ with $c\wedge-b\leq i$. The element $c'=c\wedge-i$ belongs to $\below C b$ and satisfies $\pi(c')=\pi(c)$. Since $C\leq_\sigma B$, the ideal on the right has a countable cofinal set, proving the claim.

Enumerate $B=\{b_\xi:\xi<\kappa\}$. Construct a continuous increasing chain $(C_\xi)_{\xi\leq\kappa}$ in $\FF$ with $C_0=H$, $b_\xi\in C_{\xi+1}$ and $C_{\xi+1}$ countably generated over $C_\xi$. At successors use Proposition~\ref{prop:posex}(h), and at limits use Definition~\ref{def:skeleton}(b). Then $C_\kappa=B$.

By the claim, $\pi[C_\xi]\leq_\sigma B/I$. Also, $\pi[C_{\xi+1}]$ is countably generated over $\pi[C_\xi]$, so each successor step of the continuous chain $(\pi[C_\xi])_{\xi\leq\kappa}$ is a posex. Thus $B/I$ is tightly $\sigma$-filtered over $\pi[H]$.

Finally, every Boolean algebra of size at most $\aleph_1$ is tightly $\sigma$-filtered: write it as a continuous chain of countable subalgebras and use Proposition~\ref{prop:posex}(c). Apply this to $\pi[H]$, then use Proposition~\ref{prop:concatenate}.
\end{proof}

\begin{proof}[Proof of Theorem~\ref{thm:B}]
Proposition~\ref{prop:quotient} makes $B/I$ tightly $\sigma$-filtered. The CSP of $B$ and the $P$-ideal property of $I$ then allow us to apply Lemma~\ref{lem:lifting}.
\end{proof}

\begin{remark}
Under the hypotheses of Question~\ref{q:quotient}, $B/I$ is a retract of $B$ if and only if $B/I$ is tightly $\sigma$-filtered. One implication follows from Lemma~\ref{lem:lifting}; the other follows because retracts of tightly $\sigma$-filtered algebras are tightly $\sigma$-filtered~\cite[Corollary 2.8]{Geschke}. Thus Question~\ref{q:quotient} asks whether the conclusion of Proposition~\ref{prop:quotient} remains valid for arbitrary $P$-ideals when $B$ has the CSP. In the positive case, Lemma~\ref{lem:lifting} gives the stronger conclusion that the given quotient map splits.
\end{remark}

\section{Open questions}\label{sec:questions}
The following two questions concerning the BFPP remain.
\begin{enumerate}[label=\textup{(\arabic*)}]
\item Is there a model of set theory in which $C(\omega^*)$ has the ball fixed point property?
\item Is there a model in which $\Kcc$ is not a continuous retract of $\omega^*$? If so, does $C(\omega^*)$ nevertheless fail the BFPP in that model?
\end{enumerate}
Theorem~\ref{thm:A} shows that any such model must fail~\eqref{eq:TSF}. Corollary~\ref{cor:MA} supplies models of this failure with $\cont=\aleph_2$, namely models of $\MA+\cont=\aleph_2$. The failure of tight $\sigma$-filteredness alone does not answer either question. The algebraic lifting problem in Question~\ref{q:quotient} also remains open beyond the case covered by Theorem~\ref{thm:B}.

\section{Declaration of generative AI use}
The main mathematical results were obtained by the authors through conventional mathematical research, without the use of generative AI. During the preparation and revision of this manuscript, the authors used Anthropic's Claude Opus 5.5 and OpenAI's ChatGPT for assistance with the organization and exposition of arguments, literature checks, drafting and editorial revision. All AI-assisted material retained in the manuscript was independently reviewed and verified by the authors, who assume full responsibility for the accuracy, originality and integrity of the work.


\begin{thebibliography}{99}
\bibitem{Arens} R. Arens, \emph{Representation of $*$-algebras}, Duke Math. J. \textbf{14} (1947), no.~2, 269--282.
\bibitem{AB} A. Avil\'es and C. Brech, \emph{A Boolean algebra and a Banach space obtained by push-out iteration}, Topology Appl. \textbf{158} (2011), 1534--1550.
\bibitem{BFPP} A. Avil\'es, M. Jap\'on, C. Lennard, G. Mart\'inez-Cervantes and A. Stawski, \emph{The ball fixed point property in spaces of continuous functions}, Rev. R. Acad. Cienc. Exactas F\'is. Nat. Ser. A Mat. RACSAM \textbf{120} (2026), article~136. \href{https://doi.org/10.1007/s13398-026-01934-1}{doi:10.1007/s13398-026-01934-1}.
\bibitem{Browder} F. E. Browder, \emph{Nonexpansive nonlinear operators in a Banach space}, Proc. Nat. Acad. Sci. U.S.A. \textbf{54} (1965), 1041--1044.
\bibitem{DH} A. Dow and K. P. Hart, \emph{Applications of another characterization of $\beta\omega\setminus\omega$}, Topology Appl. \textbf{122} (2002), no.~1--2, 105--133.
\bibitem{FGS} S. Fuchino, S. Geschke and L. Soukup, \emph{On the weak Freese--Nation property of $\Pow(\omega)$}, Arch. Math. Logic \textbf{40} (2001), no.~6, 425--435.
\bibitem{FKS} S. Fuchino, S. Koppelberg and S. Shelah, \emph{Partial orderings with the weak Freese--Nation property}, Ann. Pure Appl. Logic \textbf{80} (1996), no.~1, 35--54.
\bibitem{Geschke} S. Geschke, \emph{On tightly $\kappa$-filtered Boolean algebras}, Algebra Universalis \textbf{47} (2002), no.~1, 69--93.
\bibitem{GoebelKirk} K. Goebel and W. A. Kirk, \emph{Classical theory of nonexpansive mappings}, in: Handbook of Metric Fixed Point Theory (W. A. Kirk and B. Sims, eds.), Kluwer Academic Publishers, Dordrecht, 2001, 49--91.
\bibitem{Gohde} D. G\"ohde, \emph{Zum Prinzip der kontraktiven Abbildung}, Math. Nachr. \textbf{30} (1965), 251--258.
\bibitem{Kirk} W. A. Kirk, \emph{A fixed point theorem for mappings which do not increase distances}, Amer. Math. Monthly \textbf{72} (1965), 1004--1006.
\bibitem{Koppelberg} S. Koppelberg, \emph{Applications of $\sigma$-filtered Boolean algebras}, in: Advances in Algebra and Model Theory (M. Droste and R. G\"obel, eds.), Algebra Logic Appl., vol.~9, Gordon and Breach, Amsterdam, 1997, 199--213.
\bibitem{vanMill} J. van Mill, \emph{An introduction to $\beta\omega$}, in: Handbook of Set-Theoretic Topology (K. Kunen and J. E. Vaughan, eds.), Chapter~11, North-Holland, Amsterdam, 1984.
\bibitem{Parovicenko} I. I. Parovi\v cenko, \emph{A universal bicompact of weight $\aleph$}, Soviet Math. Dokl. \textbf{4} (1963), 592--595.
\bibitem{Soardi} P. M. Soardi, \emph{Existence of fixed points of nonexpansive mappings in certain Banach lattices}, Proc. Amer. Math. Soc. \textbf{73} (1979), 25--29.
\bibitem{Steprans} J. Stepr\=ans, \emph{Topological invariants in the Cohen model}, Topology Appl. \textbf{45} (1992), no.~2, 85--101.
\end{thebibliography}
\end{document}